\documentclass[12pt,reqno]{amsart}
\usepackage{psfrag}
\usepackage{tikz}
\usetikzlibrary{arrows}

\psfrag{A}{\fontsize{24}{24}$u_0^\ve$}
 \psfrag{B}{\fontsize{24}{24}$\ov{u_0^\ve}$}

\psfrag{1}{\fontsize{16}{16}$k_n$}
\psfrag{2}{\fontsize{16}{16}$k_{n-1}$}
\psfrag{3}{\fontsize{16}{16}$k_{n-2}$}
\psfrag{n}{\fontsize{16}{16}$k_1$}
\psfrag{o}{\fontsize{24}{24}$\mathcal{O}_\varepsilon$}
\psfrag{j}{\fontsize{20}{20}$t$}
\psfrag{q}{\fontsize{16}{16}$v + \xi / 2$}
\psfrag{w}{\fontsize{16}{16}$v - \xi / 2$}
\psfrag{r}{\fontsize{16}{16}$v - \xi / 2 + \sum_{l=1}^n k_l$}
\psfrag{0}{\fontsize{16}{16}$0$}
\psfrag{e}{\fontsize{16}{16}$\varepsilon$}
\psfrag{-e}{\fontsize{16}{16}$-\varepsilon$}
\psfrag{c}{\fontsize{16}{16}$v + \xi / 2$}
\psfrag{d}{\fontsize{16}{16}$v - \xi / 2$}
\psfrag{t1}{\fontsize{16}{16}$t_1$}
\psfrag{t2}{\fontsize{16}{16}$t_2$}
\psfrag{O}{\fontsize{16}{16}$O$}
\psfrag{A}{\fontsize{14}{14}$\om > \f 4 {\gamma^2} \f {\mu + 1} \mu$}
\psfrag{B}{\fontsize{14}{14}$\om = \f 4 {\gamma^2} \f {\mu + 1} \mu$}
\psfrag{C}{\fontsize{14}{14}$\om < \f 4 {\gamma^2} \f {\mu + 1} \mu$}
\psfrag{D}{\fontsize{14}{14}$\om_1$}
\psfrag{E}{\fontsize{14}{14}$\om_2$}

\newtheorem{theorem}{Theorem}[section]
\newtheorem{lemma}[theorem]{Lemma}

\theoremstyle{definition}

\theoremstyle{remark}

\newtheorem{remark}[theorem]{Remark}
\newcommand{\bes}{{\begin{split}}}
\newcommand{\ees}{{\end{split}}}
\newcommand{\bees}{{\begin{equation}\begin{split}}}
\newcommand{\es}{{\end{split}\end{equation}}}
\newcommand{\erre}{{\mathbb R}}

\newcommand{\comple}{{\mathbb C}}

\newcommand{\bea}{\begin{eqnarray}}

\newcommand{\eea}{\end{eqnarray}}

\newcommand{\be}{\begin{equation}}

\newcommand{\ee}{\end{equation}}

\newcommand{\f}{\frac}

\newcommand{\ve}{\varepsilon}

\newcommand{\om}{\omega}

\newcommand{\ov}{\overline}

\numberwithin{equation}{section}

\begin{document}
\large

\title{Lack of ground state for NLS on bridge-type graphs}


\author{Riccardo Adami, Enrico Serra, Paolo Tilli}

\address{Dipartimento di Scienze Matematiche G.L. Lagrange, 
Politecnico di Torino, Italy}




\date{April 16, 2014}

\dedicatory{Contribution for the proceedings of the Workshop on
  ``Mathematical Technology of Networks'', ZiF Bielefeld, 4-7 December
  2013.}


\tikzstyle{nodo}=[circle,draw,fill,inner sep=0pt,minimum size=\widthof{k}]
\tikzstyle{infinito}=[circle,inner sep=0pt,minimum size=0mm]

\begin{abstract} 
We prove the nonexistence of ground states for NLS on 
bridge-like graphs,
i.e. graphs with two
halflines and four vertices, of which two at infinity, 
with Kirchhoff matching conditions.
By ground
state we mean any minimizer of the energy functional among all functions
with the same mass.
\end{abstract}

\maketitle

\section{Introduction}

A {\em graph} $ {\mathcal G}$  consists of a set $V$ of points $v_1, \dots,
v_{N_v}$ called {\em vertices} and a set $A$ of {\em edges} $e_1, \dots
e_{N_e}$ joining pairs of vertices.
Multiple
connections between the same couple of vertices (i.e. several edges
between the same vertices) and also edges connecting a vertex with
itself, called {\em self-loops}, are allowed. We assume that the
number $N_v$ of vertices as well as the number $N_e$ of edges, are finite.

\noindent
We require ${\mathcal G}$ to be a {\em metric} graph, that is
we identify every edge with a real interval, namely
$$
e_j \ \longmapsto \ I_j : = (0, l_j)
$$
with $l_j \in (0, + \infty]$.

\noindent

Notice that a given vertex $v$ can act both as the left endpoint for an
edge and as the right endpoint for another one. It is then meaningful
to define on the set $A$ the functions $\sc R$ and $\sc L$, such that
${\sc L} (e_j) = v$ if $v$ is the left endpoint of $e_j$, and  ${\sc R}
(e_j) = v$ if $v$ is the right endpoint of $e_j$.

\noindent
Owing to the metric structure, it is natural to define functions 
 $u : {\mathcal G} \longrightarrow \comple$ as
$$ u : = ( u_1, \dots u_{N_e}),$$
where $u_j : I_j  \longrightarrow \comple$ is the restriction of
$u$ to the edge $e_j$. The definition of $u$ is made
complete by specifying the value of $u$ at any vertex of
$\mathcal G$.

\noindent
We define $L^p$ spaces on $\mathcal G$ according to the norm
\be \nonumber
\| u \|_{L^p ({\mathcal G})}^p \ = \ \sum_{k=1}^{N_e} \| u_k \|_{L^p (I_k)}^p. 
\ee
(In the following we will use the shorthand notation $\| u \|_p \ =
\ \| u \|_{L^p ({\mathcal G})}$ and $\| u_j \|_p \ =
\ \| u_j \|_{L^p (I_j)}$.)

\noindent
Analogously, we define the space $H^1 ({\mathcal G})$ as the subspace
of $L^2({\mathcal G})$ consisting of functions $u$ such that $ u' : =
( u_1', \dots u_{N_e}')$ is an element of  $L^2({\mathcal G})$
too, and satisfies the {\em continuity condition} at vertices, that
states that the limit of $u (x)$ as $x$ approaches a vertex $v$
exists and is independent of the particular edge on which $x$ runs.

Owing to the previous definitions, one can introduce the {\em energy
  functional}
\be \label{energy} 
E (u, {\mathcal G}) \ = \ \f 1 2 \| u ' \|_2^2 - \f 1 p  \| u
 \|_p^p, 
\ee
defined on any function $u \in H^1 ({\mathcal G})$. It is well-known that
this  functional corresponds to the conserved energy of the equation
$$
i \partial_t u (t) = - \Delta u (t) - | u (t) |^{p-2} u (t),
$$
i.e., a nonlinear Schr\"odinger equation on $\mathcal G$ with 
nonlinearity power $p-1$ (for a general introduction to NLS see \cite{c03}). 
In such equation multiplication and powers are to be understood
componentwise (i.e. edge by edge), while the definition of the Laplacian has to be 
completed by Kirchhoff boundary conditions (see the end of this
Section, and \cite{ks99} for the classification of all self-adjoint
vertex conditions).

\noindent
There is nowadays a huge literature on {\em quantum graphs}, i.e. metric
graphs with differential or pseudo-differential operators acting on
functions defined on it (\cite{n85,bcfk06,beh08,e08,k04,k05}). However, most of
such a literature is concerned with linear systems, while for
nonlinear systems the research started more than two decades ago
(\cite{am94}), got relevant results, but has remained less extensive
(\cite{bc08,ch10,cm07,cms13,Sob10}). Important results have been obtained
for dispersive estimates, that are often used to link linear and
nonlinear evolutions (\cite{b11,b12}).

\noindent

NLS-type equations are currently used to model several
systems where the propagation of waves in branched structures is
relevant: Bose-Einstein condensates in ramified traps, optical fibers, T-junctions and
others. 
In all applications it proves important to get information on
stationary solutions (i.e. the modes of the system) and on their
stability. The stationary solutions for which one
can typically state a stability result are the {\em ground states}
of the systems, i.e., the minimizers, possibly under suitable
constraints, of the functional \eqref{energy} (\cite{cl82,w86,gss87}). Observe indeed that
this functional is not  bounded from below, since, for all non-trivial
$u$, $E (\lambda u)\to -\infty$
 as $\lambda \to +\infty$. 
However, as soon as the nonlinearity power $p$ is {\em
  subcritical}, i.e.
\be \nonumber 
2 \ < \ p \ < 6, 
\ee
the restriction of $E$ to the
manifold of functions $u$ sharing the same, fixed, value $\mu$
for the
{\em mass}, namely the constraint
\be \nonumber 
\| u \|_2^2 \ = \ \mu > 0,
\ee
is bounded from below.
Indeed, by the Gagliardo-Nirenberg inequality
$$ \| u \|_p \ \leq \ C \, \| u \|_2^{\f 12 + \f 1p}  \| u \|_{H^1}^{\f 12 - \f 1p}$$ 
that can be easily
extended to graphs with a finite number of edges, one has
$$
E (u, {\mathcal G}) \ \geq \ \f 1 2 \| u ' \|_2^2  - C   \| u ' \|_2^{\f p
  2 - 1} - C
$$
where the  mass constraint was taken into account, and
lower boundedness immediately follows. So
the following questions arise.
\begin{itemize}
\item[(i)]  Is the infimum of $E$ on the space $H^1_\mu = \{
u \in H^1, \ \| u \|_2^2 = \mu \}$ larger or smaller than the infimum on
the line? 
\item[(ii)]  Is the infimum attained?
\end{itemize}
Question (ii) can be rephrased in a more physical language, as
follows: does there exist a ground state?

\noindent

The answers to (i) and (ii)
depend on the nature of $\mathcal G$, as the following
examples illustrate.
\begin{enumerate}
\item ${\mathcal G} = \erre$. The set of minimizers is given by
the {\em soliton}
\be \nonumber 
\phi_\mu (x) \ = \ C \mu^{\f 2 {6-p}} {\rm{sech}}^{\f 2 {p-2}} (c
\mu^{\f {p-2} {6-p}} x),
\ee
where $C$ and $c$ are constants depending on $p$ only,
and by the orbit of $\phi_\mu$ with respect to translations and multiplication by a
phase. Namely, the only minimizers are given by the functions
$$e^{i \theta} \phi_\mu (\cdot - y), \qquad  \theta \in [0, 2 \pi), \,
  y \in \erre$$ (\cite{cl82,gss87,gss90}).

\begin{remark} \label{inf}
By this classical result one immediately has that, if at least one edge of
$\mathcal G$ is
infinite, then $\inf_{H_\mu^1} E (u, \mathcal G) \leq E (\phi_\mu,
\erre)$. Indeed, assuming that first edge is infinite, consider
the functions
$$
u_{(n)}(x) : = ( A_n \chi_+(x) \phi_{\mu} (x - n), 0, \dots, 0 )
$$
where $\chi_+$ is a smooth function, supported on $\erre^+$, with
$\chi_+ (x) = 1$ for all $x > 1$, and $A_n$ are constants
such that  $\| u_{(n)} \|_2^2 = \mu$. Then, it is easily seen that 
$E (u_{(n)},{\mathcal G})$ converges to $E (\phi_\mu, \erre)$ as $n$ goes
to infinity, so that $\inf_{H_\mu^1} E (u, \mathcal G) \leq E (\phi_\mu,
\erre)$.
\end{remark}

\item  ${\mathcal G} = \erre^+$. In this case the only 
positive minimizer is given
  by the ``half-soliton'', i.e., by the  restriction of $\phi_{2\mu}$ 
to the positive halfline. Any further minimizer can be obtained 
by multiplying $\phi_{2\mu}$ by a phase factor.

\item ${\mathcal G} = {\mathcal S}_{n, \infty}$ with $n \geq 3$, 
 namely the star-graph made up of $n \geq 3$ halflines (in Figure 1
 the case $n=3$ is plotted). In that
  case
$$ \inf_{u \in H^1_\mu} E (u,  {\mathcal S}_{n, \infty}) \ = \ E
  (\phi_\mu, \erre) $$
but the infimum is not achieved (\cite{acfn12}).
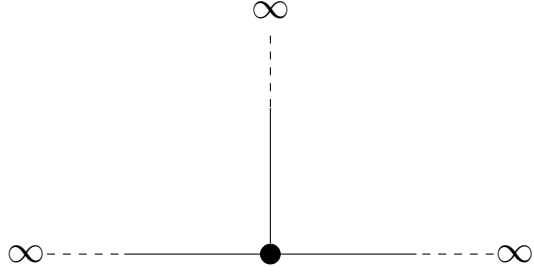
\begin{figure}[h] 
\begin{center}
\begin{tikzpicture}
[scale=1.3,style={circle,
              inner sep=0pt,minimum size=7mm}]
\node at (-2.5,0) [infinito]  (1) {$\infty$};
\node at (0,0) [nodo] (2) {};
\node at (2.5,0) [infinito]  (3) {$\infty$};
\node at (0,2.5) [infinito]  (4) {$\infty$};
\node at (1.5,0)  [minimum size=0pt] (5) {};
\node at (-1.5,0) [minimum size=0pt] (6) {}; 
\node at (0,1.5) [minimum size=0pt] (7) {};

\draw[dashed] (1) -- (6);
\draw[dashed] (3) -- (5);
\draw[dashed] (7) -- (4);
\draw [-] (2) -- (6) ;
\draw [-] (5) -- (2) ;
\draw [-] (2) -- (7) ;

 %
\end{tikzpicture}
\end{center}
\caption{The three-star graph ${\mathcal S}_{3, \infty}$.}
\end{figure}

\item ${\mathcal G} = {\mathcal B}_3$, i.e. the three-bridge graph
  portrayed in
  Figure 2.
\begin{figure}[h] \label{b3}
\begin{center}
\begin{tikzpicture}
[scale=1.3,style={circle,
              inner sep=0pt,minimum size=7mm}]
\node at (-2,0) [infinito]  (1) {$\infty$};
\node at (0,0) [nodo] (2) {};
\node at (1,0) [nodo] (3) {};
\node at (3,0) [infinito]  (4) {$\infty$};
\node at (1.7,0)  [minimum size=0pt] (5) {};
\node at (-0.7,0) [minimum size=0pt] (6) {};
\draw[dashed] (6) -- (1);
\draw [-] (6) -- (2);
\draw [-] (5) -- (3) ;
\draw [-] (2) -- (3) ;
\draw[dashed] (5) -- (4); 
\draw [-] (2) to [out=-40,in=-140] (3);
 \draw [-] (2) to [out=40,in=140] (3);
\end{tikzpicture}
\end{center}
\caption{The three-bridge ${\mathcal B}_3$.}
\end{figure}
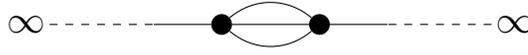

\noindent
This graph is Eulerian, i.e. it can be unfolded into a
  line, as shown in Figure 3. 
\begin{figure}[h] \label{b3u}
\begin{center}
\begin{tikzpicture}
[scale=1.3,style={circle,
              inner sep=0pt,minimum size=7mm}]
\node at (-2,0) [infinito]  (1) {$\infty$};
\node at (0,0) [minimum size=0pt] (2) {};
\node at (1,0) [minimum size=0pt] (3) {};
\node at (3,0) [infinito]  (4) {$\infty$};
\node at (1.7,0)  [minimum size=0pt] (5) {};
\node at (-0.7,0) [minimum size=0pt] (6) {};

\draw[rounded corners=8pt]
(6) -- (2)   to [out=40,in=140] (3) -- (0.1,0) to [out=-40,in=-140]
(3) -- (5);


\draw[dashed] (6) -- (1);
\draw [dashed] (5) -- (4); 
  %
\end{tikzpicture}
\end{center}
\caption{The unfolded three-bridge ${\mathcal B}_3$.}
\end{figure}
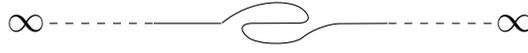
Correspondingly, every $u \in H^1_\mu ({\mathcal G})$ 
unfolds
into
  a function $\widetilde u \in H^1_\mu (\erre)$ such that $E(u,
  {\mathcal B}_3) = E (\widetilde u, \erre)$. Notice that the Eulerian
  path on $\mathcal B_3$ crosses every vertex three times, so that the
  unfolded function 
$\widetilde u$ must assume three times the values at vertices. This
  implies that $\widetilde u$ cannot be a soliton,
so the infimum
  cannot be attained.

\item ${\mathcal G} = {\mathcal B}_2$, i.e. the two-bridge graph in
  Figure 4.

\noindent
This time, the graph is not Eulerian, so the
  problem is not immediate to solve. We will show in the next section
  that the situation is exactly the same as in the previous example:
  $\inf_{H^1_\mu} E (u, {\mathcal B}_2) = E (\phi_\mu, \erre)$ but the
  infimum is not attained. The same holds for any $2k$-bridge, and
  this is the main result of this note.

\begin{figure}[h] \label{2b}
\begin{center}
\begin{tikzpicture}
[scale=1.3,style={circle,
              inner sep=0pt,minimum size=7mm}]
\node at (-2,0) [infinito]  (1) {$\infty$};
\node at (0,0) [nodo] (2) {};
\node at (1,0) [nodo] (3) {};
\node at (3,0) [infinito]  (4) {$\infty$};
\node at (1.7,0)  [minimum size=0pt] (5) {};
\node at (-0.7,0) [minimum size=0pt] (6) {};

\draw[dashed] (5) -- (4);
\draw[dashed] (6) -- (1);
\draw [-] (2) -- (6) ;
\draw [-] (5) -- (3) ;
 \draw [-] (2) to [out=-40,in=-140] (3);
 \draw [-] (2) to [out=40,in=140] (3);

\end{tikzpicture}
\end{center}
\caption{The two-bridge ${\mathcal B}_2$.}
\end{figure}
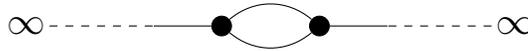

\item ${\mathcal G} =  {\mathcal S}_{2+1}$, i.e. the star-graph
  consisting of two infinite and one finite edge, displayed in Figure
  5. 
In this case $\inf_{H^1_\mu}
 E (u,  {\mathcal S}_{2+1}) <   E (\phi_\mu, \erre)$ and the infimum
 is attained, so it is actually a minimum. This result will be proved in
 the forthcoming paper \cite{ast14}.
\begin{figure}[h] \label{2+1} 
\begin{center}
\begin{tikzpicture}
[scale=1.3,style={circle,
              inner sep=0pt,minimum size=7mm}]
\node at (-2.5,0) [infinito]  (1) {$\infty$};
\node at (0,0) [nodo] (2) {};
\node at (2.5,0) [infinito]  (3) {$\infty$};
%
\node at (1.5,0)  [minimum size=0pt] (5) {};
\node at (-1.5,0) [minimum size=0pt] (6) {}; 
\node at (0,1.5) [nodo] (7) {};

\draw[dashed] (1) -- (6);
\draw[dashed] (3) -- (5);
\draw [-] (2) -- (6) ;
\draw [-] (5) -- (2) ;
\draw [-] (2) -- (7) ;

 %
\end{tikzpicture}
\end{center}
\caption{The three-star graph ${\mathcal S}_{2+1}$ with two infinite
  and one finite edge.}
\end{figure}
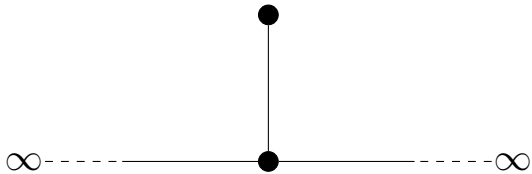

\item The {\em exceptional graph} ${\mathcal E}_3$ displayed in Figure
  6. In this case $\inf_{H^1_\mu} E (u, {\mathcal B}_3) = E
  (\phi_\mu, \erre)$ and the minimum is attained. Details will be
  given in \cite{ast14}.

\begin{figure}[h] \label{e3}
\begin{center}
\begin{tikzpicture}
[scale=1.3,style={circle,
              inner sep=0pt,minimum size=7mm}]
\node at (-3,0) [infinito]  (1) {$\infty$};
\node at (0,0) [nodo] (2) {};
\node at (3,0) [infinito]  (3) {$\infty$};
\node at (2,0)  [minimum size=0pt] (5) {};
\node at (-2,0) [minimum size=0pt] (4) {};
\node at (0,1) [nodo] (6) {};
\node at (0,2.4) [nodo] (7) {};
\draw[dashed] (4) -- (1);
\draw [-] (4) -- (2);
\draw [-] (5) -- (2) ;
\draw[dashed] (5) -- (3); 
\draw (0,0.5) circle (0.5cm);
\draw (0,1.7) circle (0.7cm);
\draw (0,2.7) circle (0.3cm);
  %
\end{tikzpicture}
\end{center}
\caption{The exceptional graph ${\mathcal E}_3$. Edges connecting the
  same couple of vertices have the same length.}
\end{figure}
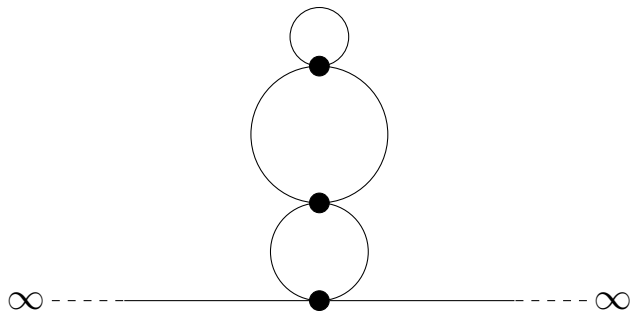

\end{enumerate}

\bigskip

In this note we treat the case of the $n$-bridge graphs ${\mathcal B}_n$,  i.e. a
graph consisting of two 
halflines whose origins are connected by $n$ finite edges (not necessarily of the same length).


We prove the
following
\begin{theorem} \label{theo}
Let $\mathcal B_n$, $n \geq 2$, be an $n$-bridge graph. Consider
the energy functional $E$ defined in \eqref{energy} with $2 < p < 6$.
Then,

$(a)$ $\inf_{H^1_\mu} E (u, \mathcal B_n) \ = \ E (\phi_\mu, \erre)$.

$(b)$ The infimum is not attained.
\end{theorem}

This is the first result on the minimization of NLS energy on
non-star graphs. A more general result, including  cases
where the infimum is attained, will be proved in \cite{ast14}. 

In order to illustrate the physical 
meaning of the absence of the ground state,
 consider for instance the case of
a Bose-Einstein condensate in a ramified trap with two long
branches. Under the critical temperature, a macroscopic fraction of
the
 particles of the system
is known to collapse in the ground state of the Gross-Pitaevskii
functional (i.e. the energy $E$ with $p=4$). In absence of a
ground state, one could imagine the system  that follows a
minimizing sequence. Of course, an actual trap will
always be finite and therefore a ground state will exist.
Nevertheless, provided that two branches exhibit a larger lengthscale than
the rest of the graph and that some other technical hypotheses are
fulfilled, the ground state should not be sensitively
different from a soliton escaping along one of the branches.

Before proving Theorem \ref{theo}, let us comment on the matching conditions at
vertices. Even though our  nonexistence result holds for bridges only,
the argument 
we give for vertex conditions is general, see also \cite{ast14}. 

Any minimizer is a stationary point for the
unconstrained functional
$$
\widetilde E (u) \ = \ E (u) + \nu \| u \|_2^2,
$$
where $\nu$ is a Lagrange multiplier. Now, since $\widetilde E$ is differentiable on $H^1({\mathcal G})$,
\begin{equation*} 
\begin{split}
\nabla \widetilde E (u)\eta \ = \ & \Re \int_{\mathcal G}  \left( \bar u'\eta'  - |u|^{p-2} \bar u \eta+ 2
\nu \bar u \eta \right) dx 
\\
 \  =  \ & \sum_{j=1}^{N_e} \Re \int_{I_j}
 \left( \bar u_j' \eta_j'  - |u_j|^{p-2} \bar u_j \eta_j + 2 \nu \bar u_j 
\eta_j \right) dx =0
\end{split}
\end{equation*}
for all $\eta\in H^1({\mathcal G})$.
By standard arguments (integrating by parts and using the
Euler--Lagrange equation in each interval), the preceding identity
yields 
\begin{equation*} \begin{split}
\Re \sum_{j=1}^{N_e}  \bar u'_j \eta_j |_{0}^{l_j}=0.
\end{split}
\end{equation*}
Focusing on vertices instead of edges, this can be equivalently written as
\begin{equation*} \begin{split}
\Re \sum_{k=1}^{N_v} \eta (v_k)  \left( \sum_{{\rm{R}} (e_j) = v_k}
\bar u'_j (l_j) - \sum_{{{\rm{L}}} (e_j) = v_k}
\bar u'_j (0) \right) =0.
\end{split}
\end{equation*}
Finally, by the arbitrariness of $\eta$, one concludes
$$
\sum_{{\rm{{R}}} (e_j) = v_k}
u'_j (l_j) - \sum_{{\rm{L}} (e_j) = v_k}
u'_j (0) \ = \ 0,\qquad\text{ for all } k,
$$ 
which are the well-known Kirchhoff conditions.

\section{Proof}
We start by giving a lemma that compares the contributions of two
different edges to the energy and shows how to construct a third edge
and a function which, properly inserted in the graph, makes the energy
decrease.
Theorem \ref{theo}
then follows as an easy consequence.

\begin{lemma}[Comparison] \label{comparison}
For $i = 1,2$, let $l_i$ be an element of $(0, + \infty]$ and denote
  by $I_i$ the
interval $(0, l_i)$.

Given a pair of functions $u_i \in H^1 (I_i) \backslash \{ 0 \}$, there exist an
interval $I = (0, l)$ with  $l \in (0, + \infty]$ and a function $w
  \in H^1 (I)$, such that
\begin{enumerate}
\item $\| w \|^2_{L^2 (I)} \ = \ \| u_1 \|^2_{L^2 (I_1)} + \| u_2 \|^2_{L^2 (I_2)}.$
\item For either $i = 1$ or $i = 2$, $w(0) = u_i (0)$ and $w(l) = u_i
  (l_i).$
\item $E (w,I) \, \leq \, E (u_1, I_1) + E (u_2, I_2)$.
\end{enumerate}
Furthermore, 
$$
 E (w,I) \, < \, E (u_1, I_1) + E (u_2, I_2),
$$
unless $u_1 = u_2= c$ for some constant $c$.

\end{lemma}

\begin{proof}
Set 
$$\lambda : = \f {\int_0^{l_2} |u_2(x)|^2 \, dx}{\int_0^{l_1} |u_1(x)|^2 \,dx},
$$ 
and define
$
{\widetilde{l}}_1 : =  (1+ \lambda) l_1$ if $l_1$ is finite, $l_2 : = \f {1 +
  \lambda} \lambda \, l_2$ if $l_2$ is finite,  
$\widetilde l_i : = + \infty$ if $l_i = + \infty$, and $\widetilde I_i
: = (0, \widetilde l_i)$.
Consider
 the functions $\widetilde u_i: \widetilde I_i \to \erre$
defined by
\be \nonumber 
\widetilde u_1 (x)  \ : = \ u_1 \left(  \f x {1 + \lambda} \right), \quad
  \widetilde u_2 (x)  \ : = \ u_2 \left(  \f {\lambda x} {1 + \lambda} \right).
\ee

\noindent
An elementary computation gives
\begin{equation} \label{elementary} \begin{split}
\int_0^{\widetilde l_1} | \widetilde u_1 ' (x) |^2 \, dx \ = & \ \f 1
    {1+\lambda} \int_0^{l_1} | u_1 ' (x) |^2 \, dx \\
\int_0^{\widetilde l_1} | \widetilde u_1  (x) |^q \, dx \ = & \ 
    {(1+\lambda)} \int_0^{l_1} | u_1  (x) |^q \, dx \\
\int_0^{\widetilde l_2} | \widetilde u_2 ' (x) |^2 \, dx
\ = & \ \f \lambda
    {1+\lambda} \int_0^{l_2} | u_2' (x) |^2 \, dx \\
\int_0^{\widetilde l_2}  | \widetilde u_2  (x) |^q \, dx
\ = & \ 
    \f {1+\lambda} \lambda \int_0^{l_2} | u_2  (x) |^q \, dx,
\end{split} \end{equation}
for any $q > 0$. Setting $q=2$, and owing to the definition of
$\lambda$, one immediately finds that for both $i=1,2$
\be \label{preservednorm}
\| \widetilde u_i \|^2_{L^2 (\widetilde I_i)} \ = \ \| u_1 \|^2_{L^2 (I_1)} + \|
u_2 \|^2_{L^2 (I_2)}
\ee
and
\be \label{preserveddir}
\widetilde u_i (0) = u_i (0), \quad \widetilde u_i (\widetilde l_i) = u_i (l_i).
\ee

\noindent
If either $u_1$ or $u_2$ is nonconstant, then
by \eqref{elementary} with $q = p$ one obtains
\be \begin{split} \label{strict} &
{E} (\widetilde u_1, \widetilde I_1) + \lambda \, {
  E} (\widetilde u_2,  \widetilde I_2) \
\\ = & \ \f 1   {2(1+\lambda)} \int_0^{l_1} | u_1 ' (x) |^2 \, dx - 
 \f {1+\lambda} p \int_0^{l_1} | u_1  (x) |^p \, dx \\ & \
+ \f {\lambda^2} {2(1+\lambda)}  \int_0^{l_2} | u_2' (x) |^2 \, dx
 - 
 \f {1+\lambda} p \int_0^{l_2} | u_2  (x) |^p \, dx \\ & \\
< & \ (1 + \lambda) \left( {E} (u_1,I_1) +   {E} (u_2,I_2) \right).
\end{split}
\ee
Then, for either $i=1$ or $i=2$ one gets 
\be \nonumber 
{E} (\widetilde u_i, \widetilde I_i) <  {E} (u_1,I_1) +
{E} (u_2,I_2).\ee
Denote this index by $\bar \i$ and define $w =
 u_{\bar \i}$,
$I = \widetilde I_{\bar \i}$. By \eqref{preservednorm},
\eqref{preserveddir}, and 
\eqref{strict}, items (1), (2), and (3) with the strict inequality
are proved for $u_1$ and 
$u_2$ that are not both constant.

\noindent
Finally, let us suppose that
$u_i \equiv \bar u_i$ for both $i=1,2$, where $\bar u_i$ is a constant. Then,
from \eqref{elementary} one has 
$$
{E} (\widetilde u_1, \widetilde I_1) \ = \ (1 + \lambda) \, {E} (u_1, I_1), \quad 
{E} (\widetilde u_2, \widetilde I_2) \ = \ \f{1 + \lambda} \lambda {E} (u_2, I_2)
$$
thus
$$
{E} (\widetilde u_1, \widetilde I_1) + \lambda \, {E} (\widetilde  u_2, \widetilde I_2)
= (1 + \lambda) \left(  {E} (u_1, I_1) +  {E} (u_2, I_2)
\right).
$$
As a consequence,
either ${E} (\widetilde u_1,  \widetilde I_1) < {E} (u_1, I_1) +
{E}  (u_2,  I_2)$ or  ${E} (\widetilde u_2,  \widetilde I_2) < {E} (u_1, I_1) +
{E}  (u_2,  I_2)$
unless ${E} (\widetilde u_1, \widetilde I_1) = {E} (\widetilde u_2,
\widetilde I_2) =  {E} (u_1, I_1) +
{E} (u_1, I_2)$.
By a straightforward computation, one finds that this implies $\bar u_1 =
\bar u_2$. The proof is complete.

\end{proof}

Now we are ready to prove Theorem \ref{theo}.

\begin{proof}[Proof of Theorem \ref{theo}]
Consider the $n-bridge$ graph ${\mathcal B}_n$  and a function $u \in H^1_\mu
({\mathcal B}_{n})$. If $n$ is odd, then ${\mathcal B}_n$
is Eulerian, so that the function $u$ unfolds to a function
$\widetilde u : \erre
\longrightarrow \comple$ s.t. 
$$ E (u, {\mathcal B}_n ) \ = \ E (\widetilde u, \erre) \ \geq \ E
(\phi_\mu, \erre)
$$
and the last inequality is an identity only if $\widetilde u (x)=
\phi_\mu ( x - y)$ for some $y$. But this is not possible, since
any value attained by $u$ at a vertex is attained at least $n$ times by 
$\widetilde u$. As $n\ge 3$, $\widetilde u$ cannot be equal to a
soliton. As a consequence,
$$  
E (u, {\mathcal B}_n ) \ > \ E (\phi_\mu, \erre).
$$
This inequality, together with Remark \ref{inf}, proves Theorem
\ref{theo} when $n$ is odd.  

When $n$ is
even, let $e_i$, $1 \leq i \leq n$, be the $i$-th edge between
the two halflines. As stated in Section 1, an interval
$I_i = (0, l_i)$ is associated  with the edge $e_i$. Focusing on 
$e_{1}, e_2$, 
by Lemma \ref{comparison} there exist an interval $I : = (0, l)$ and a function
$w_2 : I \longrightarrow \comple$ such that $\| w_2
\|_2^2 = \| u_{1} \|_2^2 + \| u_{2} \|_2^2$, $w_{2}(0) = u (v_1)$
and $w_{2}(l) = u (v_2)$, where $v_1$ and $v_2$ are the two vertices
corresponding to the origins of the two halflines. Then, the function
$$ w \ = \ (w_2, u_3, \dots , u_{n}, u_{n+1}, u_{n+2}),$$
where $u_{n+1}$ and $u_{n+2}$ are the components of $u$ on the two halflines,
is an element of $H^1_\mu ({\mathcal B}_{n-1})$. Furthermore, owing to
point (3) in 
Lemma \ref{comparison} again, one gets
\begin{equation*} \begin{split}
E (w,{\mathcal B}_{n-1}) \ = \ & E (w_2,I) + \sum_{j=3}^{n+2} E (u_j,
I_j) 
\ \leq \ \sum_{j=1}^{n+2} E (u_j,
I_j) \\
\ = \ & E (u, {\mathcal B}_{n}) .
\end{split} \end{equation*}
Since $n-1$ is odd, one concludes 
\begin{equation} \nonumber \begin{split}
E (\phi_\mu, \erre) \ < \ & E (w,{\mathcal B}_{n-1}) \ \leq \
 E (u, {\mathcal B}_{n})
\end{split} \end{equation}
and the proof is complete.

\end{proof}

\section{Possible extensions and perspectives}

The reduction technique described in the preceding section can be extended to treat more
general graphs.

For instance, a self-loop attached to an edge can be melted in a single edge, as illustrated in the following lemma.

\begin{lemma}[Removing self-loops] \label{selfloops}
Let $l_1>0$  and $l_2 \in (0,+\infty]$ and denote
  by $I_i$ the
interval $(0, l_i)$ and by $I$ the interval $(0,l_1+l_2)$.

Given a pair of functions $u_i \in H^1 (I_i)$,
with $u_1 (0) = u_1 (l_1) = u_2 (0),$ there exists  a function $w
  \in H^1 (I)$, such that
\begin{enumerate}
\item $\| w \|^2_{L^2 (I)} \ = \ \| u_1 \|^2_{L^2 (I_1)} + \| u_2 \|^2_{L^2 (I_2)}.$
\item $w(0) = u_1 (0)$ and $v(l) = u_2
  (l_2).$
\item $E (w,I) \, = \, E (u_1, I_1) + E (u_2, I_2)$.
\end{enumerate}
\end{lemma}
\begin{proof}
It is sufficient to define  $w$ on $I$ as 
\be \nonumber
w (x)  =\begin{cases} u_1 (x) & \text{ if } x \in (0,l_1) \\ 
 u_2 (x-l_1) & \text{ if } x \in (l_1, l_1 + l_2). \end{cases}
\ee
\end{proof}

Lemmas \ref{comparison} and \ref{selfloops} can be used 
in order to develop a ``haircut'' strategy suitable to work
on a larger class of graphs. Indeed, consider a graph $\mathcal
G$ with $N_e$ edges and  a function $u$ in $H^1_\mu ({\mathcal G})$. 
In several cases one may use  Lemma \ref{comparison} or Lemma \ref{selfloops} 
to construct a graph ${\mathcal G}'$ with $N_e - 1 $
edges and a function $w$ in $H^1_\mu ({\mathcal G}')$ such that
$$
E (w, {\mathcal G'}) \ \leq \ E (u, {\mathcal G}).
$$
This could be the starting point of an inductive procedure aimed at reducing any graph (by removing one edge at a time) to 
simpler graphs that one is able to handle explicitly.
This is exactly described in the forthcoming paper \cite{ast14}.





\end{document}